\documentclass[12pt, a4paper]{amsart}

\usepackage{latexsym, graphicx}

\usepackage{amssymb}
\usepackage[centertags]{amsmath}
\usepackage{verbatim}
\usepackage{epic, eepic}

\usepackage{color}

\newtheorem{theorem}{Theorem}
\newtheorem{lemma}[theorem]{Lemma}
\newtheorem{definition}[theorem]{Definition}
\newtheorem{remark}[theorem]{Remark}

\numberwithin{theorem}{section}

\renewcommand{\phi}{\varphi}

\renewcommand{\epsilon}{\varepsilon}

\newcommand{\arctanh}{\rm arctanh}

\begin{document}

\title[Magnetic billiards]
      {On Totally integrable magnetic billiards on constant curvature surface}

\date{9 August 2012}
\author{Misha Bialy}
\address{School of Mathematical Sciences, Raymond and Beverly Sackler Faculty of Exact Sciences, Tel Aviv University,
Israel} \email{bialy@post.tau.ac.il}
\thanks{Partially supported by ISF grant 128/10}

\subjclass[2000]{ } 

\begin{abstract}We consider billiard ball motion in
a convex domain of a constant curvature surface influenced by the
constant magnetic field. We prove that if the billiard map is
totally integrable then the boundary curve is necessarily a circle.
This result is a manifestation of the so-called Hopf rigidity
phenomenon which was recently obtained for classical billiards on
constant curvature surfaces.
\end{abstract}

\maketitle



\section{Introduction and the result}
Let $S$ be a surface of constant curvature $K=0, \pm 1$. Let
$\gamma$ be a simple closed convex curve on $S$ of class $C^2$. We
shall denote by $k$ the geodesic curvature of $\gamma$ and assume it
is strictly positive everywhere. We consider the so-called magnetic
billiard inside $\gamma$ where the magnitude of the magnetic field
is assumed to be constant $\beta\geq 0$. This means that the
billiard ball between elastic reflections from the boundary moves
with a unite speed along curves of constant geodesic curvature
$\beta$. The model of magnetic billiard was extensively studied (see
incomplete list
\cite{BR}\cite{Berg}\cite{Ta}\cite{T}\cite{GT}\cite{T2}\cite{GB}\cite{GSG}\cite{GE}).
Let me summarize the basic facts on the magnetic billiards which
will be omitted. First of all the main geometric assumption which
assures that the dynamics is well defined is the following
$$\beta < \min_{x\in \gamma}k(x),$$
saying that the field is not too large relative to the geodesic
curvature of the boundary. In this case magnetic billiard ball map
defines a smooth map $T$ of the phase cylinder cylinder
$\Omega=\gamma\times(0,\pi)$, where we shall denote by $x\in [0,P)$
the arc-length coordinate along $\gamma$ and $\phi\in(0,\pi)$ is the
inward angle. Moreover $T$ is a symplectic twist map, the form
$dx\wedge d(\cos\phi)$ is preserved. Remarkably, this is the same
form which appears for classical billiards. We shall denote
$d\mu=\sin\phi dx d\phi$ the invariant measure.

An important question starting from \cite{BR} is when magnetic
billiard map is integrable. The only known example of integrable
magnetic billiard is the circular billiard, in contrast to the
classical case where for any constant curvature surface ellipses are
integrable also (see \cite{V}).

We shall adopt the following definition of suggested by Andreas
Knauf for geodesic flow on the torus (\cite{K}):

\begin{definition}
The billiard ball map $T$ is called totally integrable if through
every point of the phase cylinder $\Omega=\gamma\times(0,\pi)$
passes a closed non-contractible curve which is invariant under the
map $T$.
\end{definition}
Our main result is the following theorem:
\begin{theorem}
\label{main} If the magnetic billiard map is totally integrable then
$\gamma$ must be a circle.
\end{theorem}
\begin{remark}
 A more general result can be proved using the notion of conjugate
points of twist maps (\cite {B},\cite{B1}). The more general
statement is the following: any magnetic billiard on a constant
curvature surface $S$ which has no conjugate points is circular
billiard.
\end{remark}
\begin{remark}
 In view of the previous remark one can consider this result as a
 magnetic
billiard analog of Hopf's theorem on tori without conjugate points.
It was proved in \cite{B3} that Hopf type rigidity holds true also
for magnetic geodesic flows on tori, provided the metric is
conformally flat. Notice that the magnetic field in \cite{B3} is not
assumed to be constant. In higher dimensions it is not known if the
conformal flatness assumption can be relaxed.
\end{remark}

\section{Magnetic versions of Santalo and mirror formula}
 One of the key observations for the result of theorem \ref{main} is the fact that the classical
Santalo formula  for geodesics (for the proof see \cite{Ch}) remains
the same for magnetic geodesics as follows (it was known already to
Santalo for horocycles;we need a very particular case, and refer to
\cite{GT} and \cite{T}) for the proof of the most general case):

\begin{lemma}{("Magnetic" Santalo formula)}
Let $l(x,\phi)$ be the length of the magnetic geodesic starting at
the point $x \in \gamma$ with the inward angle $\phi \in (0,\pi)$
with the boundary. Then the integral over the phase cylinder with
respect to the invariant measure $d\mu=dx\ d(-\cos \phi)$
$$\int l(x,\phi)d\mu=2\pi A,$$
independently of the magnitude of the magnetic field $\beta$. Here
$A$ is the area of the billiard domain.
\end{lemma}

Next recall the Mirror formula for usual billiard on a surface $S$
of constant curvature $K$. It reads
$$\frac {Y^{'}}{Y}(a)+\frac{Y^{'}}{Y}(b)=\frac{2k(x)}{\sin\phi},$$
where $Y$ denotes the orthogonal Jacobi field along geodesics on the
surface $S$ satisfying initial conditions $Y(0)=0, Y^{'}(0)=1$. Here
$x$ is a point on the mirror $a$ is a distance from a point $A$
inside the domain to the point $x$ along the shortest ray and $b$ is
a distance along the reflected ray to the point $B$ where the
focusing of the reflected beam occurs, $\phi$ is the angle of
reflection.

It is well known that the presence of the magnetic field results in
adding to the curvature $K$ the term $\beta^2$, so that the Jacobi
field $Y$ should be changed in the Mirror formula to $Y_{\beta}$
where:
\begin{eqnarray*}
Y_{\beta}&=&\frac{1}{\sqrt{K+\beta^2}}\sin(\sqrt{K+\beta^2}t),\  for
\ K+\beta^2>0, \\
Y_{\beta}&=&t, \ for\
K+\beta^2=0,\\
Y_{\beta}&=&\frac{1}{\sqrt{-(K+\beta^2})}\sinh(\sqrt{-(K+\beta^2})t),
 \ for \ K+\beta^2<0.
\end{eqnarray*}

 There is also a change on the right hand side
 of the Mirror formula so that for any $K=0,\pm1$ the formula reads
 as follows:
\begin{equation}
\label {mequation} \frac
{Y_{\beta}^{'}}{Y_{\beta}}(a)+\frac{Y_{\beta}^{'}}{Y_{\beta}}(b)=\frac{2(k(x)-\beta\cos\phi)}{\sin\phi}
\end{equation}

The main step in the proof of theorem \ref{main} is the reduction to
the case of non-magnetic billiard on surface which was treated in
\cite{B}. This is done in the following way. First exactly as it was
for non-magnetic billiards we have

\begin{theorem}
\label{a} If the billiard is totally integrable (or more generally
has no conjugate points), then there exists a measurable  function
on the phase cylinder $a: \Omega\rightarrow\mathbf{R}$ such that
$0<a(x,\phi)<l(x,\phi)$ which satisfies the mirror equation:
\begin{equation} \frac{Y_{\beta}^{'}}{Y_{\beta}}(a(x,\phi))+\frac{Y_{\beta}^{'}}{Y_{\beta}}(
l(x_{-1},\phi_{-1})-a(x_{-1},\phi_{-1}))=\frac{2(k(x)-\beta\cos\phi)}{\sin\phi},
\label{mirror}
\end{equation}
where $l(x,\phi)$ denotes the length of the magnetic geodesic
segment which starts at the point $x$ of $\gamma$ with the inward
angle $\phi)$.
\end{theorem}

The proof of this theorem is analogouse to the non-magnetic case and
it is omitted.

In the sequel we shall distinguish between the cases of the Plane,
$K=0$; of the Sphere $K=1$ and the Hyperbolic plane $K=-1$. In the
last case three subcases naturally appear: $\beta>1$; $ \beta=1$ and
$\beta\in(0,1)$.

\section{Planar and Spherical magnetic billiards}
For the Plane and the Sphere the mirror equation reads:

\begin{multline}
\label{mirrorplane} \sqrt{K+\beta^2}\left(\cot
(\sqrt{K+\beta^2}a(x,\phi))+\cot
(\sqrt{K+\beta^2}(l(x_{-1},\phi_{-1})-a(x_{-1},\phi_{-1}))\right)= \\
=\frac{2(k(x)-\beta \cos\phi)}{\sin{\phi}}
\end{multline}

Notice that the geometric assumption $\beta < \min_{x\in
\gamma}k(x)$ implies that the right hand side is always positive and
hence
$$\sqrt{K+\beta^2}(a(x,\phi)+l(x_{-1},\phi_{-1})-a(x_{-1},\phi_{-1}))<\pi$$
so that the lemma of \cite {B} can be applied to get the inequality

\begin{multline*}
\sqrt{K+\beta^2}\cot
\frac{\sqrt{K+\beta^2}[a(x,\phi)+l(x_{-1},\phi_{-1})-a(x_{-1},\phi_{-1})]}{2}\leq
 \\ \leq \frac
{k(x)-\beta\cos\phi}{\sin\phi}.
\end{multline*}


This can be written in equivalent way:

\begin{equation}\label{ala}
 a(x,\phi)+l(x_{-1},\phi_{-1})-a(x_{-1},\phi_{-1})\geq
\frac{2}{\sqrt{K+\beta^2}}\arctan\frac
{\sqrt{K+\beta^2}\sin\phi}{k(x)-\beta\cos\phi}
\end{equation}

Integrate (\ref {ala}) with respect to the invariant measure  $d
\mu=\sin\phi\ dxd\phi$. We get

\begin{equation}
\label{independent} \int l\ d\mu\ \geq\ \int_0^P dx\int_0^{\pi}
\frac{2}{\sqrt{K+\beta^2}}\arctan\frac
{\sqrt{K+\beta^2}\sin\phi}{k(x)-\beta\cos\phi}\ \sin\phi d\phi.
\end{equation}
For a given $x$ denote by $I(x)$ the inner integral in the right
hand side of (\ref{independent}). Then we have:
\begin{lemma}
The integral $I(x)$ in (\ref{independent}) does not depend on
$\beta$ and equals to $2\pi A(x)/P$, where $A(x)$ is the Area of the
circle on the surface $S$ having geodesic curvature $k(x)$.
\end{lemma}
\begin{proof}

 By "magnetic" Santalo formula, the integral on the left hand side
equals $2\pi A$ independently of the magnetic field $\beta$. I claim
that this fact implies without any additional calculations  that the
inner integral on the right hand side is independent on $\beta$
also. Indeed, if the boundary curve is a circle of constant geodesic
curvature $k$ on $S$ then one can easily see that there is equality
in (\ref{independent}). Moreover due to the rotational symmetry
(\ref{independent}) leads to the following equality for the circle
of curvature $k$:

$$2\pi A= P
\int_0^{\pi}\frac{2}{\sqrt{K+\beta^2}}\arctan\frac
{\sqrt{K+\beta^2}\sin\phi}{k-\beta\cos\phi}\ \sin\phi d\phi,
$$

So that the inner integral in (\ref{independent}) equals $2\pi/A(x)$
independently of $\beta$. This proves the claim. (Of course one
could compute for any $\beta$ the integral, but this is challenging
even for MATHEMATICA.)
\end{proof}
 Using the independence on $\beta$ we can compute the right hand side of
(\ref{independent}) putting $\beta=0$. But then the inequality
becomes identical to one obtained in a non-magnetic case \cite{B}.
Namely consider first the case of the Sphere, $K=1$. We have
\begin{multline}
2\pi A \geq\int_0^P dx\int_0^{\pi}2
\arctan\left(\frac{\sin\phi}{k(x)}\right)\sin\phi\ d\phi=\\
=4\int_0^P k(x)\int_0^{\pi/2}\frac{\cos^2\phi}{k^2(x)+\sin^2\phi}\
d\phi=\\
=2\pi\int_0^P(\sqrt{k^2(x)+1}-k(x))dx.
\end{multline}
This inequality implies  \cite{B} that $\gamma$ must be a circle.
This done by the following argument: Use Gauss-Bonnet to write it in
the form
$$
A\geq \int_0^P\sqrt{k^2(x)+1}\ dx-(2\pi-A),
$$
which leads to
$$
\int_0^P\sqrt{k^2(x)+1}\ dx\leq 2\pi.
$$
Denote this integral by $I$. On the other hand by Cauchy Schwartz
one has
$$
\int_0^P(\sqrt{k^2(x)+1}+1)\ dx\  \cdot \int_0^P(\sqrt{k^2(x)+1}-1)\
dx\geq\left(\int_0^P k(x)\ dx\right )^2=(2\pi-A)^2.
$$
This can be rewritten as
$$
(I-P)(I+P)\geq (2\pi-A)^2,
$$
and since $I\leq2\pi$ then
$$4\pi ^2 \geq I^2\geq P^2+A^2-4\pi A+4\pi ^2.
$$ Thus we end with the inequality
$$
0\geq P^2+A^2-4\pi A
$$ which is opposite to the isoperimetric on the sphere.

For the Plane, $K=0$, the inequality (\ref{independent}) looks even
simpler when one passes to the limit $\beta\rightarrow 0$:
$$2\pi A\geq\pi\int_0^P\frac{1}{k(x)}dx,$$
which is possible only for circles as was observed in \cite{W}.
Because by Cauchy-Schwartz one has
$$\int_0^P\frac{1}{k(x)}dx\geq\frac{P^2}{\int_0^P k(x)dx}=\frac{P^2}{2\pi} ,$$
contradicting the isoperimetric inequality in the plane.

\section{Magnetic billiards on the Hyperbolic plane}
On the Hyperbolic plane, $K=-1$, we shall proceed in a similar way
as before, dividing between the following cases where the mirror
formula (\ref{mequation}) looks differently depending on the
magnitude of the magnetic field:

Case1. Assume here that $\beta>1$. In this case the mirror equation
(\ref{mirrorplane}) looks exactly as in the Spherical and planar
case (\ref{mirrorplane}) but with $K=-1$. In this case using again
the lemma of \cite{B} we get
\begin{multline}
\label{h1} \sqrt{\beta^2-1}\cot
\frac{\sqrt{\beta^2-1}[a(x,\phi)+l(x_{-1},\phi_{-1})-a(x_{-1},\phi_{-1})]}{2}\leq
 \\ \leq \frac
{k(x)-\beta\cos\phi}{\sin\phi}.
\end{multline}Or equivalently

\begin{equation}\label{alah1}
 a(x,\phi)+l(x_{-1},\phi_{-1})-a(x_{-1},\phi_{-1})\geq
\frac{2}{\sqrt{\beta^2-1}}\arctan\frac
{\sqrt{\beta^2-1}\sin\phi}{k(x)-\beta\cos\phi}
\end{equation}

Integrate (\ref {alah1}) with respect to the invariant measure  $d
\mu=\sin\phi\ dxd\phi$. We get

\begin{equation}
\label{independenth1} \int l\ d\mu\ \geq\ \int_0^P dx\int_0^{\pi}
\frac{2}{\sqrt{\beta^2-1}}\arctan\frac
{\sqrt{\beta^2-1}\sin\phi}{k(x)-\beta\cos\phi}\ \sin\phi d\phi.
\end{equation}

Denote $I_1(x)$ the inner integral of (\ref{independenth1})

Case2. In this case $\beta=1$. In this case the effective curvature
$K+\beta^2$ vanishes and theorem \ref{a} implies
$$\frac{1}{a(x,\phi)}
+\frac{1}{l(x_{-1},\phi_{-1})-a(x_{-1},\phi_{-1})}=\frac
{2(k(x)-\cos\phi)}{\sin\phi}.
$$
Using the convexity of the function $\frac{1}{x}$ we get
\begin{equation*}
 \frac{2}{a(x,\phi)+l(x_{-1},\phi_{-1})-a(x_{-1},\phi_{-1})}\leq
\frac{(k(x)-\cos\phi)}{\sin\phi}
\end{equation*}
So that
\begin{equation}
\label{alah2} a(x,\phi)+l(x_{-1},\phi_{-1})-a(x_{-1},\phi_{-1})\geq
\frac{2\sin\phi}{k(x)-\cos\phi}
\end{equation}
Integrating with respect to the invariant measure we end up with the
inequality:
\begin{equation}
\label{independenth2} \int l\ d\mu\ \geq\ \int_0^P dx\int_0^{\pi}
\frac{2\sin\phi}{k(x)-\cos\phi}\ \sin\phi d\phi.
\end{equation}

Denote $I_2(x)$ the inner integral of (\ref{independenth2}).

Case3. In this last case $\beta\in (0,1)$. By theorem \ref{a} we
have
\begin{multline}
\label{mirrorhyp} \sqrt{1-\beta^2}\left(\frac{ \coth
(\sqrt{1-\beta^2}a(x,\phi))+\coth
(\sqrt{1-\beta^2}(l(x_{-1},\phi_{-1})-a(x_{-1},\phi_{-1}))}{2}\right)= \\
=\frac{k(x)-\beta \cos\phi}{\sin{\phi}}
\end{multline}

One can see that for the given $x$ the minimum of the right hand
side of (\ref{mirrorhyp}) equals $\sqrt{k^2-\beta^2}$ which is
attained for some angle $\phi\in (0,\pi)$. Comparing with the left
hand side, which is obviously strictly greater than
$\sqrt{1-\beta^2}$, we get
$$\sqrt{1-\beta^2}<\sqrt{k^2-\beta^2},$$
or equivalently
\begin{equation}
\label{horo}
 k(x)\geq 1
 \end{equation}
 So that
 the curve $\gamma$ must be convex with respect to horocycles.

Moreover, by the convexity of $coth$
\begin{multline}
\label{h3} \sqrt{1-\beta^2}\coth
\frac{\sqrt{1-\beta^2}[a(x,\phi)+l(x_{-1},\phi_{-1})-a(x_{-1},\phi_{-1})]}{2}\leq
 \\ \leq \frac
{k(x)-\beta\cos\phi}{\sin\phi}.
\end{multline}Or equivalently

\begin{equation}\label{alah3}
 a(x,\phi)+l(x_{-1},\phi_{-1})-a(x_{-1},\phi_{-1})\geq
\frac{2}{\sqrt{1-\beta^2}}\arctanh \frac
{\sqrt{1-\beta^2}\sin\phi}{k(x)-\beta\cos\phi}
\end{equation}

Integrate (\ref {alah3}) with respect to the invariant measure  $d
\mu=\sin\phi\ dxd\phi$. We get

\begin{equation}
\label{independenth3} \int l\ d\mu\ \geq\ \int_0^P dx\int_0^{\pi}
\frac{2}{\sqrt{1-\beta^2}}\arctanh\frac
{\sqrt{1-\beta^2}\sin\phi}{k(x)-\beta\cos\phi}\ \sin\phi d\phi.
\end{equation}

Denote $I_3(x)$ the inner integral of (\ref{independenth3})

Remarkably the following lemma holds true:
\begin{lemma}
All three integrals $I_1, I_2, I_3$ are independent on $\beta$ and
$$I_1(x)=I_2(x)=I_3(x)=2\pi(k(x)-\sqrt{k^2(x)-1}).$$
\end{lemma}
\begin{proof}
This goes exactly like in the Spherical case. Indeed take a circle
on the Hyperbolic Plane of curvature $k=k(x)$, notice that such a
circle exists since in the first two Cases $k(x)>1$ just by the
geometric assumption and in Case3 this is obtained in (\ref{horo}).

For any circle both inequalities
(\ref{independenth1},\ref{independenth3}) becomes equalities. So
using rotational symmetry  both integrals $I_1,I_3$ can be easily
computed to be equal $\frac{2\pi A(x)}{P}$ which shows independence
of $\beta$. Moreover, it is clear that for $\beta\rightarrow1$ both
integrals $I_1,I_3$ tend to $I_2$, which can be easily computed.
This completes the proof of the lemma.
\end{proof}

It is easy to finish the section.
 All three inequalities of the Cases1,2,3
(\ref{independenth1}),(\ref{independenth2}),(\ref{independenth3})
lead by the "magnetic" Santalo formula and by the lemma to the same
inequality which does not contain $\beta$ anymore. We proceed like
in \cite{B}.
$$
A\geq\int_0^P( k(x)-\sqrt{k^2(x)-1})\ dx.
$$
Use Gauss-Bonnet to write it in the form
$$
A\geq 2\pi+A-\int_0^P\sqrt{k^2(x)-1}\ dx
$$
and therefore
$$
\int_0^P\sqrt{k^2(x)-1}\ dx \geq 2\pi
$$
On the other hand the last integral can be estimated from above by
the Cauchy-Schwartz
$$
\int_0^P\sqrt{k^2(x)-1}\ dx\leq \left(\int_0^P (k(x)-1)\ dx \int_0^P
(k(x)+1)\ dx \right)^{\frac{1}{2}}=
$$
$$
=((A+2\pi-P)(A+2\pi+P))^{\frac{1}{2}}
$$
where we applied Gauss Bonnet again. Thus we have the inequality
$$((A+2\pi-P)(A+2\pi+P))^{\frac{1}{2}}\geq 2\pi$$
which is equivalent to
$$A^2+4\pi A-P^2\geq 0.$$
But this is opposite to the isoperimetric inequality on the
Hyperbolic plane. Thus $\gamma$ must be a circle. This completes the
proof for the Hyperbolic plane.

\end{document}